\documentclass[12pt]{amsart}
\usepackage{amsmath,amssymb,comment,hyperref,latexsym,
mathrsfs,todonotes}
\usepackage[all]{xy}

\theoremstyle{remark}

\newtheorem{para}{\bf}[subsection]

\newtheorem{remarks}[para]{\bf Remarks}
\newtheorem{remark}[para]{\bf Remark}
\newtheorem{conjecture}[para]{\bf Conjecture}

\theoremstyle{definition}

\theoremstyle{plain}

\newtheorem{theorem}[para]{Theorem}

\newtheorem{lemma}[para]{Lemma}
\newtheorem{corollary}[para]{Corollary}
\newtheorem{proposition}[para]{Proposition}

\newenvironment{numequation}{\addtocounter{para}{1}
\begin{equation}}{\end{equation}}

\newcommand{\bbP}{{\mathbb P}}
\newcommand{\bbQ}{{\mathbb Q}}

\newcommand{\bbZ}{{\mathbb Z}}

\newcommand{\bG}{{\bf G}}

\newcommand{\bP}{{\bf P}}

\newcommand{\cG}{{\mathcal G}}

\newcommand{\cM}{{\mathcal M}}

\newcommand{\cO}{{\mathcal O}}

\newcommand{\cS}{{\mathcal S}}

\newcommand{\ovE}{{\overline E}}

\def\Z{\mathbb Z}

\def\Q{\mathbb Q}

\def\0{\emptyset}

\def\a{\alpha}
\def\b{\beta}

\def\s{\sigma}
\def\vep{\varepsilon}
\def\vp{\varphi}

\def\ab{{\rm ab}}
\def\alg{{\rm alg}} 
 
\def\an{{\rm an}}

\def\cont{{\rm cont}}

\def\diag{\operatorname{diag}}

\def\Gal{\operatorname{Gal}}
\def\GL{\operatorname{GL}}
\def\Hom{\operatorname{Hom}}
\def\hra{\hookrightarrow}

\def\Ind{\operatorname{Ind}}
\def\ind{\operatorname{ind}}

\def\lalg{{\rm lalg}} 

\def\lra{\longrightarrow}
\def\midc{\; | \:}

\def\ophi{{\overline{\phi}}}
\def\opsi{{\overline{\psi}}}

\def\Qp{{\bbQ_p}}
\def\Qpb{{\overline{\bbQ}_p}}

\def\res{{\rm res}}

\def\sm{{\rm sm}}

\def\st{{\rm st}}
\def\St{{\rm St}}
\def\sub{\subset}
\def\supp{\operatorname{supp}}
\def\Sym{\operatorname{Sym}}

\def\triv{{\bf 1}}

\def\vpi{{\varpi}}
\def\WD{\operatorname{WD}}
\def\x{{\times}}

\title{$p$-adic Banach space representations of $SL_2(\Qp)$}

\author{Dubravka Ban}
\address{School of Mathematical and Statistical Sciences, 1245 Lincoln Drive, Southern Illinois University, Carbondale, Illinois 62901, U.S.A.}
\email{dban@siu.edu}

\author{Matthias Strauch}
\address{Department of Mathematics, Indiana University, Rawles Hall, Bloomington, Indiana 47405, U.S.A.}
\email{mstrauch@indiana.edu}

\begin{document}

\begin{abstract} We consider the restriction to $SL_2(\Qp)$ of an irreducible $p$-adic unitary Banach space representation $\Pi$ of $GL_2({\mathbb Q}_p)$. If $\Pi$ is associated, via the $p$-adic local Langlands correspondence, to an absolutely irreducible 2-dimensional Galois representation $\psi$, then the restriction of $\Pi$ decomposes as a direct sum of $r \le 2$ irreducible representations. The main result of this paper is that $r$ is equal to the cardinality $s$ of the centralizer in $PGL_2$ of the projective Galois representation $\opsi$ associated to $\psi$, and the restriction is multiplicity-free, except if $\psi$ is triply-imprimitive, in which case the restriction of $\Pi$ is a direct sum of two equivalent representations. 
\end{abstract}
\maketitle

\tableofcontents

\section{Introduction}

\addtocounter{subsection}{1}

In this paper, we derive a classification of 
absolutely irreducible admissible unitary Banach space representations of $H=SL_2(\Qp)$ (Propositions \ref{prop:ordinary} and \ref{prop:class}) from the classification of the corresponding representations of $G = GL_2(\Qp)$, cf. \cite[1.1]{ColDoPa}. The classification of the latter representations is in turn part of the $p$-adic Langlands correspondence for $GL_2(\Qp)$. ``Unitary'' refers to the condition that the group action is norm-preserving. Furthermore, throughout our paper the coefficient fields are finite extensions of $\Qp$.

Recall that an absolutely irreducible admissible Banach space representation is called {\em ordinary} if it is a subquotient of a continuous parabolic induction of a unitary character.
If $\Pi$ is an ordinary representation of $G$, we prove that  $\Pi|_H$ is absolutely irreducible.
Moreover, if $\Pi$ and $\Pi'$ are inequivalent  ordinary representation of $G$, then $\Pi|_H$
and $\Pi'|_H$ are also inequivalent. Consequently, nontrivial ordinary representations of $H$ are  classified by the inducing data, namely the unitary characters of $\Q_p^\times$
(see Proposition~\ref{prop:ordinary}).

For non-ordinary representations of $H$, 
we start with an absolutely irreducible representation
$\psi: \cG_\Qp = \Gal(\Qpb/\Qp) \to GL_2(E)$,
where $E$ is a finite extension of $\Q_p$.
By Colmez' $p$-adic Langlands correspondence there is attached to $\psi$ a unitary Banach space representation $\Pi(\psi)$ of $G$, cf. \cite[0.17]{Colmez10}.
The functor $$\psi \rightsquigarrow \Pi(\psi)$$ induces a bijection
between the set of equivalence classes of 
absolutely irreducible continuous two-dimensional $E$-representations of $\cG_\Qp$ and
absolutely irreducible non-ordinary admissible unitary $E$-Banach space representations of $GL_2(\Qp)$ \cite[1.1]{ColDoPa}.

Motivated by the description of $L$-packets for $H$
in the classical local Langlands correspondence
\cite{LL}, for a given $\psi:  \cG_\Qp \to GL_2(E)$ and
$\Pi=\Pi(\psi)$, we consider the corresponding 
projective representation $\opsi: \cG_\Qp \to PGL_2(E)$
and the restriction $\Pi|_H$.
Our main result is the following theorem.

\begin{theorem}[Theorem~\ref{Main}]\label{MainIntr}  Let $\psi: \cG_\Qp \to GL_2(E)$ be an absolutely irreducible continuous representation and $\Pi=\Pi(\psi)$.  Denote by $\opsi: \cG_\Qp \to PGL_2(E)$ the corresponding projective representation. Then, after possibly replacing $E$ by a finite extension, we have: 

\begin{enumerate}
    \item[(i)]$\Pi|_H$ is either absolutely irreducible or it decomposes as a direct sum of two absolutely irreducible representations. Let $r$ be the number of components of $\Pi|_H$ (counting multiplicities).

    \item[(ii)] Let $S_\opsi$ be the centralizer in $PGL_2(\ovE)$ of the image of $\opsi$. Then $S_\opsi$ is finite, with cardinality 1,2, or 4.

    \item[(iii)] If $\psi$ is not triply-imprimitive, then  $\Pi|_H$ is multiplicity free. Moreover,
    \[
        r=|S_\opsi| \leq 2.
    \]
    That is, the number of components of $\Pi|_H$
    equals the cardinality of $S_\opsi$.

    \item[(iv)] If $\psi$ is  triply-imprimitive, then  $|S_\opsi|=4$ and $\Pi$ restricts to $H$ with multiplicity two. More specifically,
    $\Pi|_H$  decomposes as a direct sum of two equivalent irreducible representations. \qed

\end{enumerate}

\end{theorem}

For the definition of a triply-imprimitive two-dimensional Galois representation, see \ref{defn-triply}.

\begin{remark} The cases described in Theorem~\ref{MainIntr} match the classical Langlands correspondence for $GL_1(D)$, where $D$ is the quaternion division algebra over $\Qp$ \cite[p. 751, lines 1-3, Lemma 7.1]{LL}. In particular, if $\s$ is an admissible smooth representation of $GL_1(D)$ corresponding to a triply-imprimitive representation of the Weil group, then $\s$ restricts to $SL_1(D)$
with multiplicity two. In contrast to this, restrictions of smooth irreducible representations of $G$ to $H$ are always multiplicity free \cite[Lemma 2.6]{LL}. 
\end{remark}

The proof of Theorem~\ref{MainIntr} is based on the compatibility of the $p$-adic Langlands correspondence with twisting by characters: for every continuous character $\chi$ of $\cG_\Qp$ there is an isomorphism of Banach space representations
\begin{numequation}\label{twist}
    \Pi(\chi \otimes \psi) \cong (\chi \circ \det) \otimes \Pi(\psi).
\end{numequation}
In addition, we need certain properties of the restriction of representations from $GL_2(\Qp)$ to $SL_2(\Qp)$. We obtain those as special cases of some general results about the restriction of irreducible admissible Banach space representations of $p$-adic Lie groups to subgroups of finite index. First, we consider a $p$-adic Lie group $G$ with an open normal subgroup $H$ of finite index. If $\Pi$ is an irreducible admissible $E$-Banach space representation of $G$, it is shown that $\Pi|_H$ decomposes as a direct sum of  irreducible admissible representations of $H$ (Proposition~\ref{restrictionPi}).
Next, we specialize to the case when $G/H$ is finite abelian
and $\Pi$ is absolutely irreducible. After possibly replacing $E$ by a finite extension, the description of $\Pi|_H$ parallels 
the classical case of smooth representations. 
In particular, it depends on the group of characters
\[
    X_H(\Pi) = \{ \chi \text{ a character of } G \text{ trivial on } H \mid 
    \chi \otimes \Pi \cong \Pi \;,\}
\]
see Proposition~\ref{restrictionPiAbs} and compare to  \cite[Section 2]{gk82}.
Going back to $G=GL_2(\Qp)$ and $H=SL_2(\Qp)$, 
we recall that the representation $\Pi=\Pi(\psi)$ is absolutely irreducible, and by
\cite[1.1]{DospinescuSchraen} it has thus a central character. Furthermore, the group $G/ZH$, where $Z$ is the center of $G$, is finite abelian. If $(\chi \circ \det) \otimes \Pi \cong \Pi$ the character 
$\chi \circ \det$ is necessarily trivial on $Z$. Therefore, 
the decomposition of $\Pi|_H$ depends on the finite group $X(\Pi)$ of characters $\chi \circ \det$ of $G$
such that $(\chi \circ \det) \otimes \Pi \cong \Pi$.
By \eqref{twist}, $X(\Pi)$ can be identified with the group 
$X(\psi)$ of characters $\chi$ of  $\cG_\Qp$ satisfying 
$\chi \otimes \psi \cong \psi$, and $X(\psi)$ is easily seen to be isomorphic to $S_\opsi$, the centralizer in $PGL_2(\ovE)$ of the image of $\opsi$. Thus,
the decomposition of $\Pi|_H$ can be described 
using $S_\opsi$, which leads to the proof of Theorem~\ref{MainIntr}.

In the last section of the paper we study the
connection between these results and the classical local Langlands correspondence.
Suppose $\psi: \cG_\Qp \to GL_2(E)$ is de Rham with distinct Hodge-Tate weights $a_\psi<b_\psi$. In this case, the $p$-adic Langlands correspondence $\psi \mapsto \Pi(\psi)$
encapsulates the classical Langlands correspondence. We investigate the related picture for $H=SL_2(\Qp)$.

More precisely, to $\psi$ is associated a Fontaine-Deligne module $D=D(\psi)$ and a filtration $Fil_\psi(D_L)$ on $D_L$
(see Section~\ref{sec:LocAlg} for details).
Here, $L$ is a finite Galois extension of
$\Qp$ such that $\psi$ becomes semistable when restricted to $\cG_L=\Gal(\overline{\Q}_p/L)$.
From $D$, we obtain the
corresponding Weil-Deligne representation
$\WD(\psi)$. 
Let  $\phi = \WD(\psi)^{F{\rm -s.s.}}$
 be the  $F$-semisimplification of $\WD(\psi)$, in the sense of \cite[8.5]{Deligne_fonctions_L}.
The space of locally algebraic vectors $\Pi^\lalg$ of $\Pi=\Pi(\psi)$ is nonzero and is
of the form $\Pi^\lalg=\Pi^\alg \otimes \s$,
where $\Pi^\alg$ is algebraic and 
$\s$ is smooth-admissible.
Then, $\s$ corresponds to $\phi$ 
under the local Langlands correspondence
(slightly modified from the classical correspondence---see Remark~\ref{rem:classic}).
Let $\cS_\ophi = S_\ophi/S_\ophi^\circ$ be the group of connected components of the centralizer $S_\ophi$ of the associated projective Weil-Deligne representation $\ophi$.
Let $\s|_H = \s_1 \oplus \ldots \oplus \s_{r(\phi)}$ be the decomposition of $\s|_H$ into irreducible representations of $H$. 
The representations $\s_1, \ldots, \s_{r(\phi)}$ are inequivalent, and they form an $L$-packet.
The members of the packet are parametrized by the
characters of $\cS_\ophi$.
Consequently, $|\cS_\ophi|=r(\phi)$.
We study the relation between 
 $|S_\opsi|$ and  $|\cS_\ophi|$, which also describes the relation between $\s|_H$ and $\Pi|_H$.
Notice that the same Weil-Deligne representation $\phi$ may correspond to different Galois representations $\psi$.
The full information about $\psi$ is encoded in the Fontaine-Deligne module $D$
together with the filtration $Fil_\psi(D_L)$.
Given a rank-one submodule $D^0_L$ of $D_L$, we denote by  $\Psi(D,D^0_L)$ the set of all pairs $(\psi, D(\psi) \stackrel{\sim}{\lra} D)$ of Galois representations $\psi$, together with an isomorphism $D(\psi) \stackrel{\sim}{\lra} D$ (of Deligne-Fontaine modules) such that $Fil^{-a_\psi}_\psi(D_L)$ is mapped to $D^0_L$ under (the map induced by) this isomorphism (after base change to $L$). It turns out that $|S_\opsi|$ depends only on $D$ and $D^0_L$, and it is same for all $\psi \in \Psi(D,D^0_L)$, regardless of the breaks in the filtration $Fil_\psi(D_L)$. In other words,
$|S_\opsi|$ does not depend on the Hodge-Tate weights $a_\psi$ and $b_\psi$.

The following theorem sumarizes 
the relation between 
 $|S_\opsi|$ and  $|\cS_\ophi|$.
More details can be found in Lemma~\ref{lem:triang}  and Proposition~\ref{S_opsi}.

\begin{theorem} Suppose $\psi: \cG_\Qp \to GL_2(E)$ is de Rham with distinct Hodge-Tate weights. Let  $\phi = \WD(\psi)^{F{\rm -s.s.}}$ and 
let $D$ be the Fontaine-Deligne module attached to $\psi$. 

\begin{enumerate}
    \item[(i)] If $\psi$ is trianguline, then 
    $|\cS_\ophi| \in \{1,2\}$ and
    $|S_\opsi| =|\cS_\ophi|$.

    \item[(ii)] Suppose $\psi$ is not trianguline
    and continues to be non-trianguline after base change to any finite extension of $E$. Then 
    $|\cS_\ophi| \in \{1, 2, 4\}$ and
    $|S_\opsi| \leq |\cS_\ophi|$.
    More specifically, if $|\cS_\ophi|=2$ then there are precisely four rank-one submodules $D^0_L$ of $D_L$ such that $|S_{\overline{\psi'}}| = 2$ for all $\psi' \in \Psi(D,D^0_L)$. If $|\cS_\ophi|=4$ then there are precisely six rank-one submodules $D^0_L$ of $D_L$ such that $|S_{\overline{\psi'}}| = 2$ for all $\psi' \in \Psi(D,D^0_L)$. In all other cases, $|S_{\opsi}| = 1$.
\qed

\end{enumerate}

\end{theorem}

The relation between $\Pi|_H$
and $\s|_H$ can now be easily described; this is done in 
corollaries \ref{cor:trian} and \ref{cor:nontrian}.

\subsection{Acknowledgements} 

We would like to thank Jeffrey Adams for some valuable remarks regarding $L$-packets. Furthermore, we are very grateful to Gabriel Dospinescu for answering a number of questions we had regarding his joint work  \cite{ColDospNiz} with P. Colmez and W. Nizio\l, and for several helpful comments. The work on this project was supported in part by 
D.B.'s  Collaboration Grant for Mathematicians and we thank the Simons Foundation for their support.

\subsection{Notation}

Let $\Q_p \subseteq F \subseteq E$ be a sequence of finite extensions, with the rings of integers $\Z_p \subseteq \cO_F \subseteq \cO_E$. We fix a uniformizer $\vpi_F$ of 
$F$ and denote by $v_F: F^\x \to \Z$ the normalized valuation (i.e., 
$v_F(\vpi_F) = 1$). Let $q=p^f$ be the cardinality of the residue field of 
$F$, and denote by $|\cdot| = |\cdot|_F$ the absolute value specified by $|x| = q^{-v_F(x)}$.  

All representations on topological vector spaces are tacitly assumed to be continuous. We call such a representation {\it irreducible} if it is topologically irreducible, i.e., if it is non-zero and the zero subspace is the only proper closed subspace stable under the group action. A Banach space representation is called {\it unitary} if the group action is norm-preserving. If $V$ is an $E$-Banach space representation of $G$, we denote by $V^\an$ the subspace of locally analytic vectors of $V$. Similarly, if $V$ is a Banach representation, or a locally analytic representation, $V^\sm$ is the subspace of smooth vectors and $V^\lalg$ is the subspace of locally algebraic vectors of $V$.

If $V$ is a vector space over a subfield $E' \sub E$, then we write $V_E$ for the base change $V \otimes_{E'} E$. 

Given a parabolic subgroup $P$ of $G$, with Levi decomposition $P=MU$, we will use several types of parabolic induction: 
$\ind^G_P(\cdot)$ denotes the smooth induction, $\Ind^G_{P}(\cdot)^\an$ the locally analytic  induction, and $\Ind^G_{P}(\cdot)^\cont$
the continuous induction (all non-normalized).

The absolute Galois group of $\Qp$ will be denoted by $\cG_\Qp$.

\section{Restrictions of representations of \texorpdfstring{$p$}{}-adic Lie groups}\label{Sec2}

\subsection{Restrictions of admissible Banach space representations}

Let $G$ be a $p$-adic Lie group \cite[sec. 13]{Schneider_Lie}, and let $K \sub G$ be an open compact subgroup. Define
\[
   \cO_E[[K]] = \varprojlim_N \cO_E[K/N] \quad \text{ and }
   \quad  E[[K]]= E \otimes_{\cO_E} \cO_E[[K]] \;,
\]
where $N$ runs through all open normal subgroups of $K$. These are both noetherian rings, cf. \cite[beginning sec. 3]{ST-Banach}. An $E$-Banach space representation $\Pi$
of $K$ is called {\it admissible} if the continuous dual $\Pi' = \Hom^\cont_E(\Pi,E)$ is a finitely generated $E[[K]]$-module, cf. \cite[3.4]{ST-Banach}. In the duality theory of \cite{ST-Banach} (building on Schikhof's duality \cite{Schikhof_duality}) it is important that $\Pi'$ is equipped with the bounded-weak topology. 
This is by definition the finest locally convex topology on $\Pi'$ that coincides 
with the weak topology on (norm) bounded 
subsets of $\Pi'$.

An $E$-Banach space representation of $G$ is called admissible if it is admissible as a representation of every compact open subgroup $K$ of $G$. Admissibility can be tested on a single compact open subgroup. 

If $\Pi$ is an admissible $E$-Banach space representation of $G$, then the map $\Pi_0 \mapsto \Pi_0'$ is a bijection between the set of $G$-invariant closed vector subspaces $\Pi_0 \sub \Pi$ and the set of $G$-stable $E[[K]]$-quotient modules of $\Pi'$, cf. \cite[3.5]{ST-Banach}.

\begin{proposition}\label{restrictionPi}
Let $G$ be a $p$-adic Lie group and $H$ an open normal subgroup of $G$ of finite index. Let $\Pi$ be an irreducible admissible
$E$-Banach space representation of $G$. 

\vskip8pt 

(i) $\Pi|_H$ contains closed $H$-stable subspaces which are irreducible representations of $H$. Any such subrepresentation is automatically admissible as a representation of $H$.

\vskip8pt

(ii) Let $\pi_1 \sub \Pi$ be a closed $H$-stable subspace which is irreducible as a representation of $H$. Then there are closed $H$-stable subspaces $\pi_2, \ldots, \pi_r$ of $\Pi$ such that each $\pi_i$ is an irreducible admissible representation of $H$, and the canonical map $\pi_1 \oplus \cdots \oplus \pi_r \to \Pi$ is a topological isomorphism, i.e., we have an isomorphism
\[\Pi|_H = \pi_1 \oplus \cdots \oplus \pi_r\] 
of $E$-Banach space representations of $H$. For each $i \in \{1, \ldots, r\}$ and each $g \in G$ one has $g.\pi_i  = \pi_j$ for some $j \in \{1, \ldots,r\}$, and $G$ acts transitively on the set $\{\pi_1, \ldots, \pi_r\}$. 

\vskip8pt

(iii) Every closed irreducible $H$-subrepresentation of $\Pi$ has a complement. \end{proposition}

\begin{proof} Choose a compact open subgroup $K \sub H$. As $\Pi$ is assumed to be irreducible, $\Pi \neq 0$, and hence $\Pi' \neq 0$, by \cite[3.5]{ST-Banach}. Therefore, $\Pi'$ contains proper $H$-stable $E[[K]]$-submodules (e.g., the zero submodule). Let $\cM$ be the set of proper $E[[K]]$-submodules $M \sub \Pi'$ which are $H$-stable. $\cM$ is not empty, by what we have just observed. As $\Pi$ is assumed to be admissible, the continuous dual space $\Pi'$ is a finitely generated $E[[K]]$-module. Given any ascending  chain $M_0 \sub M_1 \sub \ldots$ in $\cM$ its union is again a proper $H$-stable submodule, and thus an upper bound in $\cM$. Zorn's Lemma implies that $\Pi'$ contains a maximal $H$-invariant $E[[K]]$-submodule $M$.
The quotient $\Pi'/M$ then corresponds to a topologically irreducible closed
$H$-subrepresentation $\pi$ of $\Pi|_H$, namely $\pi = M^\perp := \{v \in \Pi \midc \forall \ell \in M: \ell(v) = 0\}$. Because $\Pi'/M$ is finitely generated as $E[[K]]$-module, $\pi$ is an admissible
representation of $H$.

For each $g \in G$, $g.\pi$ is an $H$-invariant closed vector subspace of $\Pi$.
It is irreducible as $H$-representation because if $\pi_0 \subsetneq g.\pi$ is a proper $H$-invariant closed subspace, then $g^{-1}.\pi_0 \sub \pi$ is a subspace of the same type, which must hence be zero.
Let $\{g_1, g_2, \dots, g_k\}$ be a set of coset representatives of $G/H$.
Then
\[
    \sum_{i=1}^k (g_i).\pi
\]
is a $G$-invariant subspace of $\Pi$. It is closed because it corresponds to the $G$-stable $E[[K]]$-submodule $\bigcap_{i=1}^k (g_i).M$. Hence, it is equal to $\Pi$, because $\Pi$
is irreducible. We select a minimal subset $\{g_{i_1}, g_{i_2}, \dots, g_{i_r}\}$
of $\{g_1, g_2, \dots, g_k\}$ such that 
$
	\sum_{j=1}^r (g_{i_j}).\pi = \Pi \;.
$
We claim that this sum is direct. Suppose that 
\begin{numequation}\label{intersection}
   (g_{i_l}).\pi \cap \sum_{j \ne l} (g_{i_j}).\pi \neq 0 \;.
\end{numequation}
Then, as $\sum_{j \ne l} (g_{i_j}).\pi$ is closed (by the same argument as above) and $H$-stable, and because $(g_{i_l}).\pi$ is topologically irreducible, the left-hand side of \ref{intersection} must be $(g_{i_l}).\pi$. This implies $\sum_{j \ne l} (g_{i_j}).\pi = \Pi$,
contradicting the minimality. The map $\bigoplus_{j=1}^r (g_{i_j}).\pi \to \Pi$ is thus a continuous bijection, and hence a topological isomorphism \cite[8.7]{NFA}.

Assertion (iii) is an immediate consequence of assertion (ii). \end{proof}

\begin{corollary}\label{CorRes}
Let $\Pi$ be an irreducible admissible $E$-Banach space representation of the group $GL_n(F)$. 
Assume that $\Pi$ has a central character.\footnote{This is the case when $\Pi$ is absolutely irreducible, by \cite[1.1]{DospinescuSchraen}.}
Then there are closed $H$-stable subspaces $\pi_1, \ldots, \pi_r$ of $\Pi$, each of which is an irreducible admissible $E$-Banach representation of $SL_n(F)$, and such that the canonical map $\pi_1 \oplus \cdots \oplus \pi_r \to \Pi$ is an isomorphism of topological vector spaces. The group $GL_n(F)$ acts transitively on the set $\{\pi_1, \ldots, \pi_r\}$ as in \ref{restrictionPi}.
\end{corollary}

\begin{proof} Let $Z$ be the center of $GL_n(F)$ and $H = ZSL_n(F)$. Then $H$ is an open normal subgroup of $GL_n(F)$ of finite index. The statement now follows from \ref{restrictionPi}.
\end{proof}

\begin{proposition}\label{restrictionPialg}
Let $\Pi$ be an irreducible admissible $E$-Banach space representation of $GL_n(F)$ which has a central character. Suppose that the subspace of locally algebraic vectors $\Pi^\lalg$ is dense in $\Pi$. Write $\Pi = \pi_1 \oplus \cdots \oplus \pi_r$ as in \ref{CorRes}. Then for each $i$, the subspace $(\pi_i)^\lalg$ is dense in $\pi_i$, hence non-zero, and 
\[
    (\Pi^\lalg)|_H = (\pi_1)^\lalg \oplus \cdots \oplus (\pi_r)^\lalg \;.
\]
\end{proposition}

\begin{proof} Consider $v \in V^\lalg$ and write it as
\[
     v = v_1 + \ldots + v_k \; \mbox{ with }v_i \in \Pi_i \;.
\]
Because the projection map $\Pi \to \pi_i$, $i=1,\ldots,r$, is continuous and $H$-equivariant, it follows that  $v_i \in (\pi_i)^\lalg$, for all $i$. This shows the inclusion ''$\subseteq$''. Because we assume $\Pi^\lalg$ to be non-zero, the central character must be locally algebraic. But then any vector in $\pi_i^\lalg$ is locally algebraic for $ZH$ (cf. proof of \ref{CorRes}), and thus locally algebraic for $G$, which shows ''$\supseteq$''.
\end{proof}

The next result shows that, in the general setting of \ref{restrictionPi}, every irreducible admissible Banach space representation of $H$ is an irreducible closed subspace and direct summand of the restriction of an irreducible admissible Banach space representation of $G$. 

\begin{proposition}\label{prop:induced}
Let $G$ be a $p$-adic Lie group and $H$ an open normal subgroup
of $G$ of finite index. Let $(\pi,V)$ be an irreducible admissible Banach space representation of $H$. Then there exists an irreducible admissible Banach space representation 
$\Pi$ of $G$ such that $\Pi|_H$ contains a constituent
isomorphic to $\pi$. If $\pi$ is unitary, then $\Pi$ can be chosen to be unitary as well. 
\end{proposition}

\begin{proof}
Let $\Xi = \ind_H^G(\pi)$. This is the representation of $G$
on the space
\[
   W = \{ f:G \to V \mid \forall h \in H,
   g \in G: f(hg) = \pi(h)f(g)   \}
\]
given by $\Xi(g)f(x)=f(xg)$.

We fix a set of representatives
$
\{ g_1, g_2, \dots, g_n\}
$
of $G/H$, and define a norm on $W$ by setting
\[
    \|f\| = \max_i \|f(g_i)\|
\]
for $f \in W$. Then $W$ is a Banach space with continuous $G$-action. For every $g \in G$ and $i \in \{1, \ldots, n\}$ there is $j(i,g) \in \{1, \ldots,n\}$ and $h_{i,g} \in H$ such that $g_ig = h_{i,g}g_{j(i,g)}$, and the map $i \mapsto j(i,g)$ is a permutation of $\{1, \ldots,n\}$. Now, if $\pi$ is unitary, we have
\[
\|g.f\| = \max_i \|f(g_ig)\| = \max_i \left\|f\Big(h_{i,g}g_{j(i,g)}\Big)\right\| = \max_i \|f(g_{j(i,g)})\| = \|f\| \;.\] 
This shows that $\Xi$ is unitary if $\pi$ is unitary. For $i=1, \dots, n$, set
\[
    W_i = \{f \in W \mid \supp(f) \subseteq g_i H \}.
\]
Then each $W_i$ is $H$-invariant and 
$W|_H=W_1 \oplus \cdots \oplus W_n$. 
More specifically, define $\varphi_i : W_i \to V$
by $\varphi_i(f)=f(g_i)$. Then $\varphi_i$ is a linear bijection.
It is continuous because
\[
    \|f\|=\|f(g_i)\| \quad \text{for } \, f \in W_i \;.
\]
Moreover,
\[
    \varphi_i(\pi(h)f) 
    = f(g_ih) = \pi(g_ih g_i^{-1})f(g_i)
    = g_i^{-1} \pi (h)\varphi_i(f)
\]
for all $f \in W_i$ and $h \in H$. Hence, $W_i$ is isomorphic 
to $g_i^{-1} \pi$, so
\[
    \Xi|_H = \bigoplus_{i=1}^n g_i^{-1} \pi \;.
\]
Since each constituent is admissible, we see that $\Xi|_H$
is also admissible. 
Then  $\Xi$ is admissible as well. Hence, it must contain an irreducible
subrepresentation $\Pi$. Let $(\pi_0,W_0)$ be an
irreducible constituent of $\Pi|_H$.
Since $W_0 \subset W = W_1 \oplus \cdots \oplus W_n$, there
exists an index $j$ such that the projection of $W_0$
to $W_j$ is not zero. By irreducibility, $W_0 \cong W_j$.

In conclusion, $\Pi|_H$ contains a representation
isomorphic to $g_j^{-1}\pi$, and this is a direct summand, by \ref{restrictionPi}. Then $(g_j\Pi)|_H$ contains a constituent isomorphic to $\pi$.
\end{proof}

\subsection{Restrictions of absolutely irreducible representations}

Let $G$ be a $p$-adic Lie group and $H$ an open normal subgroup of $G$ such that $C=G/H$ is finite abelian.
Denote by $X_H$ the group of characters of $G$ trivial on $H$. Then we can identify $X_H$
with $\hat{C}$, the group of characters of $C$.

In this section, we need the assumption that every irreducible character of $C$ is defined over $E$.
This is satisfied if $E$
 contains the $m$-th roots of unity, where $m$ is the
exponent of $C$, i.e., the least common multiple of the orders of all
elements of $C$. 

Let $\Pi$ be an irreducible admissible
$E$-Banach space representation of $G$. 
For the proof of Proposition~\ref{restrictionPiAbs} below,
we need Schur's lemma, which holds if and only if $\Pi$ is absolutely irreducible (see \cite[Theorem 1.1]{DospinescuSchraen}).

\begin{lemma}\label{lemma:decomp}
Let $G$ be a $p$-adic Lie group and $H$ an open normal subgroup of $G$ such that $C=G/H$ is finite abelian.
Assume that every irreducible character of
$C$ is defined over $E$.

\begin{enumerate}
    \item[(i)] Let $\pi$ be an  irreducible admissible $E$-Banach space representation of $H$. Then
    \[
    \res^G_H \circ \ind^G_H (\pi) = \bigoplus_{g \in G/H} g\pi.
    \]
    \item[(ii)] Let $\Pi$ be an irreducible admissible $E$-Banach space representation of $G$. Then
    \[
    \ind^G_H \circ \, \res^G_H (\Pi) = \bigoplus_{\chi \in X_H} \chi \otimes \Pi.
    \]

\end{enumerate}

\end{lemma}

\begin{proof} (i) Follows from proof of Proposition~\ref{prop:induced}.

(ii) We will prove
\[
    \ind^G_H \circ \, \res^G_H (\Pi) \cong (\ind^G_H 1) \otimes \Pi.
\]
Set
\[
  \begin{aligned}
  U &= \{ f: G \to E \mid f(hg)=f(g) \, \text{ for all } h \in H, g \in G \} \\
  W &= \{ F: G \to V \mid F(hg)=\Pi(h)F(g) \, \text{ for all } h \in H, g \in G \}.
  \end{aligned}
\]
Then $U$ is the space of $\ind^G_H 1$ and $W$ is the space of $\ind^G_H \circ \res^G_H (\Pi)$.
Define $\psi: U \otimes V \to W $ by
\[
    \psi(f \otimes v)(x)=f(x) \Pi(x) v,
\]
for $f \in U$, $v \in V$, and $x \in G$.
It is easy to show that $\psi(f \otimes v)$ is 
indeed an element of $W$ and that $\psi$ is an intertwining operator.
To prove that $\psi$ is bijective, we fix a set 
of representatives $\{ g_1, \dots, g_n \}$ of $G/H$.  Notice that any $f \in U$ and any $F \in W$ are completely determined by their values at $g_i$'s.
For $j \in \{1, \dots, n\}$,
define $f_j \in U$ by $f_j(g_i)=\delta_{ij}$.
Then $f_1, \dots, f_n$ is a basis of $U$ and any 
$y \in U \otimes V$ can be written in a unique way as
$y =f_1 \otimes v_1 + \cdots + f_n \otimes v_n$,
where $v_i \in V$.

Now, suppose $y \in \ker \psi$, and write 
$y =f_1 \otimes v_1 + \cdots + f_n \otimes v_n$.
Then for any $x \in G$, 
\[
   f_1(x)\Pi(x)v_1 + \cdots + f_n(x)\Pi(x)v_n=0.
\]
In particular, for  $x=g_i$,  we get $\Pi(g_i)v_i=0$, and hence $v_i=0$. 
This holds for any $i \in \{1, \dots, n\}$. It follows $y=0$, proving that $\psi$ is injective.

For surjectivity, take $F \in W$. Let 
$y = \sum_{i=1}^n f_i \otimes \Pi(g_i^{-1}) F(g_i) \in U$.
Then $\psi(y)(g_i) = F(g_i)$, for all $i$, and hence $F = \psi(y)$.

This proves 
$
    \ind^G_H \circ \, \res^G_H (\Pi) \cong (\ind^G_H 1) \otimes \Pi.
$
To conclude the proof, we observe that we can identify $\ind^G_H 1$ with $\ind^C_1 1$,
which is the right regular representation of $C$.
Since $C$ is abelian and every irreducible character of
$C$ is defined over $E$, $\ind^C_1 1$ decomposes as a direct sum of characters of $C$. 
\end{proof}

\begin{proposition}\label{restrictionPiAbs}

Let $G$ be a $p$-adic Lie group and $H$ an open normal subgroup of $G$ such that $C=G/H$ is finite abelian. Assume that  every irreducible character of
$C$ is defined over $E$.
Let $\Pi$ be an absolutely irreducible admissible
$E$-Banach space representation of $G$
and let $\pi$ be an irreducible subrepresentation
of $\res^G_H (\Pi)$. Assume that $\pi$ is also absolutely irreducible.

\begin{enumerate}
    \item[(i)] The multiplicity $m$ of $\pi$ in 
    $\res^G_H (\Pi)$ is equal to the multiplicity 
    of $\Pi$ in $\ind^G_H(\pi)$. 
    
    \item[(ii)] Define
    $
       G_{\pi} = \{ g \in G \mid g \pi \cong \pi \}.
    $  
    Then
\[
    \res^G_H(\Pi) \cong  m \bigoplus_{g \in G/G_\pi} g\pi.
\]

    \item[(iii)] Define $X_H(\Pi)= \{ \chi \in X_H \mid \chi \otimes \Pi \cong \Pi\}$. Then
    \[
    \ind^G_H(\pi) \cong m \bigoplus_{\chi \in X_H/X_H(\Pi)} \chi \otimes \Pi.
    \]

    \item[(iv)] We have $|X_H(\Pi)|=m^2|G:G_\pi|$ and $m^2|X_H:X_H(\Pi)|=|G_\pi:H|$.   
    
\end{enumerate}

\begin{proof} We can identify $X_H(\Pi)$ with  $\hat{C}(\Pi)= \{ \chi \in \hat{C} \mid \chi \otimes \Pi \cong \Pi\}$.

(i) It follows from Lemma~\ref{lemma:decomp}(ii) that 
$\ind^G_H(\pi)$ decomposes as a direct sum of representations $\chi \otimes \Pi$,
for $\chi$ in some subset of $\hat{C}$.
In addition, it can be shown using the standard arguments from representation theory that the Frobenius reciprocity holds in this context. Then,
\[
  \Hom_G^{cont}(\Pi, \ind^G_H(\pi) ) \cong \Hom^{cont}_H(\res^G_H(\Pi), \pi).
\]
Since the representations under consideration decompose as direct sums of absolutely irreducible subrepresentations, the multiplicity of $\Pi$ 
in $\ind^G_H(\pi)$ is equal to
\[
  \dim \Hom^{cont}_G(\Pi, \ind^G_H(\pi) ) = \dim \Hom^{cont}_H(\res^G_H(\Pi), \pi),
\]
and this is equal to the multiplicity of $\pi$ in 
$\res^G_H (\Pi)$.

(ii) We know that  $\res^G_H(\Pi)$ is a direct sum of 
$c\pi$, for $c$ in some subset of $C$.
By grouping equivalent components, we get 
    \[
        \res^G_H(\Pi) \cong \bigoplus_{c \in D } m_c c \pi,
    \]
for some subset $D$ of $C$, where $m_c$ is the multiplicity of $c \pi$ in $\res^G_H(\Pi).$
The action of $G$ permutes the isotypic components $m_c c\pi$. This action is transitive  because $\Pi$ is irreducible.
Since $G$ acts by conjugation, it is easy to see that   
the multiplicities $m_c$ are all equal.
This proves (ii).

(iv) The multiplicity of $\Pi$ in 
$\bigoplus_{\chi \in \hat C} \chi \otimes \Pi$ is
$|\hat{C}(\Pi)|$. On the other hand, it is equal to
\[
    \begin{aligned}
       \dim \Hom_G^{cont}(\Pi, \bigoplus_{\chi \in \hat C} \chi \otimes \Pi) 
       &=  \dim \Hom_G^{cont}(\Pi,  \ind^G_H \circ \res^G_H (\Pi)) \\
        &=  \dim \Hom_H^{cont} (\res^G_H (\Pi), \res^G_H (\Pi)) = m^2 |G:G_\pi|.
    \end{aligned}
\]
It follows $|\hat{C}(\Pi)|=m^2|G:G_\pi|$.
We can write
\begin{equation}\label{Eq:sum}
       \bigoplus_{\chi \in \hat C} \chi \otimes \Pi =
   m^2 |G:G_\pi| \bigoplus_{\chi \in \hat C/ \hat C(\Pi)} \chi \otimes \Pi.
\end{equation}
It follows $|G:H|=  m^2 |G:G_\pi| \cdot |\hat C: \hat C(\Pi)|$, so $|G_\pi:H| = m^2|\hat{C}:\hat{C}(\Pi)|$, proving (iv).

(iii)
 From (ii), we know that 
$\res^G_H (\Pi)$ can be written as a sum of 
$m |G:G_\pi|$ components $g\pi$
(counting multiplicities).
Notice that $\ind^G_H(g\pi) \cong \ind^G_H(\pi)$, for any $g \in G$.
Equation \eqref{Eq:sum} then implies
\[
    \ind^G_H(\pi) \cong m \bigoplus_{\chi \in \hat C/ \hat C(\Pi)} \chi \otimes \Pi,
\]
completing the proof.
\end{proof}

\end{proposition}

\begin{corollary}\label{cor:two}

With assumptions and notation as in Proposition~\ref{restrictionPiAbs}, assume in addition that for some prime $\ell$, $|X_H(\Pi)|= \ell^2$
and that $\Pi|_H$ has at most $\ell$ components
(counting multiplicities). Then  $\Pi|_H=\ell \pi$.

\end{corollary}

\begin{proof}  

By Proposition~\ref{restrictionPiAbs} (iv), we have $|X_H(\Pi)| = m^2 |G:G_\pi|$. Since 
$|X_H(\Pi)|=\ell^2$, it follows $m=\ell$ or
$m=1$.
 If $m = \ell$, we are done.

Assume $m=1$. Then the number of components of
$\Pi|_H$ is $|G:G_\pi|$, and by assumption,
$|G:G_\pi| \leq \ell$.
On the other hand, 
$|G:G_\pi| = |X_H(\Pi)| = \ell^2$, a contradiction.
\end{proof}

\subsection{Restrictions of continuous principal series of \texorpdfstring{$GL_n(F)$}{}}

We need the following \cite[Corollary 1.2]{BanHun}:

\begin{proposition}\label{PropHomInd}
Let $\bG$ be a split connected reductive $\Z$-group and $\bP$ a minimal parabolic subgroup
of $\bG$. Let $G_0=\bG(\cO_F)$ and $P_0=\bP(\cO_F)$.
For any two continuous characters $\chi_1$ and $\chi_2$ of $P_0$, we have
\[
    \Hom_{G_0}^\cont (\Ind_{P_0}^{G_0}(\chi_1), \Ind_{P_0}^{G_0}(\chi_2)) = 
   \begin{cases}
      0   & \text{if } \, \chi_1 \ne \chi_2, \\
   E \cdot \rm{id} & \text{if } \, \chi_1 = \chi_2.
   \end{cases}
\]

\end{proposition}

\begin{proposition}\label{prop:irred}

Let $G=GL_n(F)$ and $H=SL_n(F)$.
Let $\chi$ be a continuous character of the minimal parabolic subgroup $P$ of $G$. 
We denote by $\Ind^G_{P}(\chi)^\cont$
the continuous principal series representation.
Then $\Ind^H_{P\cap H}(\chi|_H)^\cont$ is 
reducible if and only if 
$\Ind^G_{P}(\chi)^\cont$ is reducible.

\end{proposition}

\begin{proof}

Set $\Pi = \Ind^G_{P}(\chi)^\cont$. Notice that $G/P \cong H/(H \cap P)$ and so the spaces $\Ind^G_{P}(\chi)^\cont$ and $\Ind^H_{P\cap H}(\chi|_H)^\cont$ are equal. Then as $H$-representations
\[
      \Pi|_H = \Ind^H_{P\cap H}(\chi|_H)^\cont.
\]
If $\Pi$ is reducible, then $\Pi|_H$ is clearly 
reducible.

Assume $\Pi$ is irreducible. If $\Pi|_H$ were reducible,
we would have $\dim \Hom_H(\Pi|_H, \Pi|_H) >1$
because $\Pi|_H$ decomposes as a direct sum of irreducible representations.
This is impossible by Proposition~\ref{PropHomInd}.
\end{proof}

\section{Representations of \texorpdfstring{$SL_2(\Qp)$}{}}

From now on, $G=GL_2(\Qp)$ and $H=SL_2(\Qp)$.

\subsection{Ordinary representations}

Let $T \subset G$ be the torus of diagonal matrices and $T_H=T \cap H$. Then $T\cap H$ is isomorphic to $\Q_p^\times$ via the map
$a \mapsto \diag(a, a^{-1})$.
Fix a minimal parabolic subgroup of $P$ of $G$ containing $T$. Then $P=TU$, where $U$ is the unipotent radical of $U$, and $P \cap H = T_H U$.
Given a continuous character $\chi$ of $T$ (respectively, $T_H$), we extend it trivially to $U$ to get a continuous character of $P$ (respectively, $P \cap H$), which we denote by the 
same letter $\chi$.

Recall that the smooth Steinberg representation $\St$ is irreducible as both a $G$-representation and an $H$-representation. It sits in the exact sequence
\[
   0 \to 1 \to \ind^G_P(1) \to \St \to 0.
\]
Denote by $\widehat{\St}$ the universal unitary completion of $\St$. Then $\widehat{\St}$ is  topologically irreducible and admissible as a representation of $G$ and it sits in the exact sequence
\begin{numequation}\label{contSteinberg}
0 \to 1 \to \Ind^G_{P}(1)^\cont \to \widehat{\St} \to 0 \;,
\end{numequation}
\cite[Lemma 5.3.3]{EmertonLocalGlobal}.

\begin{proposition}\label{prop:ordinary}

Ordinary representations of $H=SL_2(\Qp)$ are
$1$, $\widehat{\St}$, and $\Ind^H_{P\cap H}(\chi)^\cont$, where $\chi$ runs over all nontrivial unitary characters of $\Q_p^\times.$
All these representations are pairwise inequivalent.

\end{proposition}

\begin{proof}  

Ordinary representations of $H$ are the irreducible components of 
$\Ind^H_{P\cap H}(\chi)^\cont$, where $\chi$ runs over all  unitary characters of $\Q_p^\times.$
From \cite[Proposition 5.3.4]{EmertonLocalGlobal},
we know that for $\chi \ne 1$, the representation 
$\Ind^G_P(\chi \otimes \chi^{-1})^\cont$ is topologically irreducible.
Proposition~\ref{prop:irred} then implies that
$\Ind^H_{P\cap H}(\chi)^\cont$ is also irreducible.
All these representations are pairwise inequivalent by Proposition~\ref{PropHomInd}.

\vskip8pt

To see that $\widehat{\St}|_H$ is irreducible, we consider the exact sequence \ref{contSteinberg} and pass to locally analytic vectors to obtain the exact sequence
\begin{numequation}\label{locanSteinberg}
0 \to 1 \to \Ind^G_{P}(1)^\an \to \St^\an := \widehat{\St}^\an \to 0 \;,
\end{numequation}
cf. \cite[7.1]{ST-AlgOfDist}. The locally analytic Steinberg representation $\St^\an$ contains the smooth Steinberg representation $\St$, and one has an exact sequence
\begin{numequation}\label{Steinbergexact}
0 \lra \St \lra \St^\an \lra \Ind^H_{P_H}(\chi_{2\rho})^\an \lra 0 \;,
\end{numequation}
where the representation on the right is the locally analytic induction of the character $\chi_{2\rho}(\diag(a,d)) = \frac{a}{d}$. This latter principal series representation is (topologically) irreducible as representation of $H$, by \cite[3.1.6]{KisinStrauch} or \cite[3.5.2]{Orlik-Strauch10}. We can consider $\St^\an$ as the space of $E$-valued locally analytic functions on $G/P \cong H/P_H \cong \bbP^1(\Qp)$, modulo the space of constant functions. The map $\St^\an \to \Ind^H_{P_H}(\chi_{2\rho})^\an$ on the right of \ref{Steinbergexact} is then, in suitable local coordinates, given by differentiating functions. By \cite[4.3]{Kohlhaase-Cohomology}, this map has no continuous section. This observation, coupled with the fact that $\St$ is irreducible as $H$-representation, shows that $\St$ is the only non-zero closed proper $H$-subrepresentation of $\St^\an$.

\vskip8pt

Let $V \sub \widehat{\St}$ be a closed non-zero $H$-invariant subspace. Then the subspace of $H$-locally analytic vectors $V^\an \sub V$ is again non-zero (by, again, \cite[7.1]{ST-AlgOfDist}) and closed. It follows from what has been said above that $V^\an = \St$ or $V^\an = \St^\an$. In either case, $V^\an$ contain $\St$ which is dense in the continuous Steinberg representation $\widehat{\St}$, by definition. But then, as $V$ is supposed to be closed, $V = \widehat{\St}$. 
\end{proof}

\subsection{At most two irreducible constituents}

\begin{proposition}\hskip-6pt\footnote{We learned this fact from a comment by Matthew Emerton.}\label{Prop2components} Let $\Pi$ be an absolutely irreducible admissible unitary $p$-adic Banach space representation of $G$. Then $\Pi|_{SL_2(\Qp)}$ decomposes into at most two irreducible components.
\end{proposition} 

\begin{proof} Put $\overline{\Pi} = \Pi_{\le 1} \otimes_{\cO_E} k_E$, where $\Pi_{\le 1} = \{v \in \Pi \; | \;  \|v\| \le 1\}$ and $k_E$ is the residue field of $E$. This is a smooth $G$-representation. By \cite[1.4]{ColDoPa}, after possibly replacing $E$ by an unramified quadratic extension, there are two possibilities for $\overline{\Pi}$, namely

\vskip5pt

(i) $\overline{\Pi}$ is an absolutely irreducible supersingular representation.

\vskip5pt

(ii) The semisimplification $\overline{\Pi}^{\rm ss}$ of $\overline{\Pi}$ embeds into 
\begin{numequation}\label{PaskunasThm1.1.(ii)}   
    \pi\{\chi_1,\chi_2\} := \Big(\Ind^G_P(\chi_1 \otimes \chi_2 \omega^{-1})\Big)^{\rm ss} \oplus \Big(\Ind^G_P(\chi_2 \otimes \chi_1 \omega^{-1})\Big)^{\rm ss} \;,
\end{numequation}
where $\chi_1$ and $\chi_2$ are smooth characters $\bbQ_p^\x \to k_E^\times$, and $\omega: \bbQ_p^\x \to k_E^\x$ is the reduction of the cyclotomic character. 

\vskip5pt

It is a result of R. Abdellatif that in case (i) $\overline{\Pi}|_H$ decomposes into two irreducible representations, cf. \cite[Th\'eor\`eme 0.7]{Abdellatif}. In particular, $\Pi|_H$ cannot have more than two irreducible components.

\vskip5pt

Now suppose we are in case (ii). We consider the list given in \cite[2.14]{ColDoPa} which provides an explicit description of the decomposition of $\pi\{\chi_1,\chi_2\}$ into irreducible constituents. $\pi\{\chi_1,\chi_2\}$ is isomorphic to one (and only one) of the following:

\begin{enumerate}\item  $\ind^G_P(\chi_1 \otimes \chi_2 \omega^{-1}) \oplus \ind^G_P(\chi_2 \otimes \chi_1 \omega^{-1})$,  if  $\chi_1\chi_2^{-1} \neq \triv,\; \omega^{\pm 1}$\, ;

\vskip5pt

\item $\ind^G_P(\chi \otimes \chi \omega^{-1})^{\oplus 2}$,  if  $\chi_1 = \chi_2 = \chi$  and  $p \ge  3$;

\vskip5pt

\item $\Big(\triv \oplus \St \oplus \ind^G_P(\omega \otimes \omega^{-1})\Big) \otimes \chi \circ \det$,  if  $\chi_1 \chi_2^{-1} = \omega^{\pm 1}$ and $p \ge 5$;

\vskip5pt

\item $\Big(\triv \oplus \St \oplus \omega \circ \det \oplus \St \otimes \omega \circ \det\Big) \otimes \chi \circ \det$,  if  $\chi_1 \chi_2^{-1} = \omega^{\pm 1}$ and $p=3$;

\vskip5pt

\item $\Big(\triv \oplus \St \Big)^{\oplus 2} \otimes \chi \circ \det$, if  $\chi_1 = \chi_2$  and  $p = 2$.
\end{enumerate}

\vskip8pt

Here, $\St$ denotes the smooth Steinberg representation in characteristic $p$ which is irreducible by \cite[0.1]{Abdellatif}. The restriction of $\ind^G_P(\chi_1 \otimes \chi_2 \omega^{-1})$ to $H$ is just $\ind^H_{P_H}(\chi_1 \otimes \chi_2 \omega^{-1})$, and by \cite[0.1]{Abdellatif}, this representation is irreducible if and only if $\chi_1 \neq \chi_2 \omega^{-1}$. Thus we see that in case (1) the restriction is of length two. The same is true in case (2) because $\omega$ is not trivial when $p \ge 3$. In case (3), since $\omega^2 \neq \triv$ for $p \ge 5$, the representation $\pi\{\chi_1,\chi_2\}|_H$ has two infinite-dimensional constituents, and the same is true in cases (4) and (5), because the Steinberg representation $\St$ of $H$ is irreducible. 

\vskip8pt

Write $\Pi|_H = \Pi_1 \oplus \ldots \oplus \Pi_r$, with irreducible $H$-representations $\Pi_i$. By \ref{CorRes} the irreducible representations $\Pi_i$ are permuted by the action of $G$, and they must hence be all infinite-dimensional. Therefore, the representation $(\overline{\Pi})^{\rm ss}|_H$ must have at least $r$ infinite-dimensional irreducible constituents. By what we have just seen, $r$ can then be at most two.   
\end{proof}

\subsection{The main result}\label{subsec:main}

Let $\eta$ be a  continuous character of $G$. Since $H$ is the derived subgroup of $G$, $\eta$ is trivial on $H$. It follows that 
$\eta$ is of the form
\[
    \eta = \chi \circ \det,
\]
where $\chi$ is a continuous character of $\Q_p^\times$. We define $\cG_\Qp^*$ to be the group of continuous characters $\cG_\Qp \to \overline{\Qp}^\x$. As $\cG_\Qp$ is compact, any such character $\vep$ has image in $(E')^\x$ for a finite extension $E'$ of $\Qp$. When, in the following, we write $\psi \cong \vep \otimes \psi$, any such isomorphism is to be understood to exist over some finite extension $E'$ which contains the image of $\vep$, i.e., after possibly enlarging the coefficient field $E$ of $\psi$. We use the same convention when considering Banach space representations $\Pi$, and we set 
\[
   X(\psi)= \{\chi \in \cG_\Qp^*\midc \chi \otimes \psi \cong  \psi\}
\]
and 
\[
   X(\Pi)= \{\chi \in \cG_\Qp^* \mid (\chi \circ \det) \otimes \Pi \cong  \Pi \}.
\]

\begin{lemma}\label{lem:X}

Let $\psi: \cG_\Qp \to GL_2(E)$ be an absolutely irreducible continuous representation and $\Pi = \Pi(\psi)$. Then 
\begin{enumerate}

    \item[(i)] $X(\Pi)$ is finite and it can be identified with a subgroup of the group of characters of $G/ZH$.
    \item[(ii)] $X(\psi)=X(\Pi)$, and all nontrivial characters of these sets are of order two. 
    \item[(iii)] Denote by $\opsi: \cG_\Qp \to PGL_2(E)$ the corresponding projective representation. Let $S_\opsi$ be the centralizer in $PGL_2(\ovE)$ of the image of $\opsi$. Then $S_\opsi \cong X(\psi)$. 
   
\end{enumerate}

\end{lemma}

\begin{proof}  

(i) Take $\chi \in X(\Pi)$. Since $(\chi \circ \det) \otimes \Pi \cong  \Pi$, the representations $(\chi \circ \det) \otimes \Pi$ and $\Pi$ have the same central character. It follows that 
$(\chi \circ \det)$ is trivial on $Z$. Since it is also trivial on $H$, it factors through 
$G/ZH$. The group $G/ZH$ is finite, so it has finitely many characters. It follows that $X(\Pi)$
is finite.

(ii) Note that the existence of an isomorphism $\psi \cong \vep \otimes \psi$ implies (by taking determinants) that $\vep^2$ is the trivial character. Furthermore, the $p$-adic Langlands correspondence  
$\psi \rightsquigarrow \Pi(\psi)$
is compatible with twisting by characters, i.e., for every $\chi \in \cG_\Qp^*$ we have
\[
    \Pi(\chi \otimes \psi) \cong (\chi \circ \det) \otimes \Pi(\psi) \;,
\]
cf. \cite[III.13]{ColmezDospinescu14}. Then $\chi \in X(\psi)$ if and only if $\chi \in X(\Pi)$.

(iii) Given $\vep \in X(\psi)$, there is an isomorphism $g_\vep: \psi \to \vep \otimes \psi$. The isomorphism $g_\vep$ is an automorphism of the vector space $E^2$ underlying $\psi$, i.e., we consider $g_\vep$ as an element of $GL_2(E)$. We note that $g_\vep$ is unique up to multiplication by a non-zero scalar, because if $g'_\vep: \psi \to \vep \otimes \psi$ is another isomorphism, then $g'_\vep g_\vep^{-1}$ is an automorphism of $\psi$, and thus a scalar. We have thus a well-defined map $X(\psi) \to PGL_2(E)$, $\vep \mapsto \overline{g_\vep}$, where $\overline{g_\vep}$ is the class of $g_\vep$ in $PGL_2(E)$. It is immediate, that for $\vep, \delta \in X(\psi)$ the composition $g_\vep \circ g_\delta$ is an isomorphism $\psi \to (\vep \delta) \otimes \psi$, and therefore $\overline{g_\vep} \circ \overline{g_\delta} = \overline{g_{\vep \delta}}$. We have therefore a group homomorphism $X(\psi) \to PGL_2(E)$, whose image, as is immediately seen, lies in the centralizer $S_\opsi$.
Furthermore, this homomorphism is injective. It is also surjective, because if $\overline{g} \in S_\opsi$ is represented by the matrix $g \in GL_2(\overline{E})$, then there is a character $\vep$ such that $g\psi(x)  = \vep(x)\psi(x)g$, for all $x \in \cG_\Qp$, and then $g$ is equal to $g_\vep$, up to a non-zero scalar multiple.  
\end{proof}

\begin{proposition}\label{prop:card}
The cardinality of $X(\psi)$ divides 4.
\end{proposition}

\begin{proof} This can be shown as in the case of Weyl groups, cf. \cite[Prop. in 41.3]{BushnellHenniart_GL2}. Let us recall the arguments.\footnote{The arguments actually apply to the case of an absolutely irreducible 2-dimensional representation $\psi$ of {\it any} group $G$: the corresponding set $X(\psi)$ has always cardinality 1, 2, or 4.} If $|X(\psi)|>1$ then there is a continuous character $\vep$ of $\cG_\Qp$ and an isomorphism $f: \psi \stackrel{\simeq}{\lra} \vep \otimes \psi$ over some finite extension $E'$ of the coefficient field $E$. Comparing determinants we see that $\vep^2 = \triv$ is the trivial character, and we find that $\vep = \kappa_{F/\Qp}$ is the quadratic character whose kernel is $\cG_F$ for a quadratic extension $F/\Qp$. We claim that $\psi|_{\cG_F}$ is not absolutely irreducible: if it were, then $f$ would have to be multiplication by a non-zero scalar, but this cannot be, since $\vep = \kappa_{F/\Qp}$ is non-trivial on $\cG_\Qp$ and $\psi$ is assumed to be absolutely irreducible as $\cG_\Qp$-representation. Hence there is a one-dimensional subspace $W$ of $\psi$ which is stable under $\cG_F$, i.e., $\cG_F$ acts on $W$ by a continuous character $\xi$. It follows that $\psi \cong \ind^{\cG_\Qp}_{\cG_F}(\xi)$. We thus have $\psi|_{\cG_F} \cong \xi \oplus \xi^\sigma$, where $\sigma$ is the non-trivial Galois automorphism of $F$ over $\Qp$. Furthermore, $\ind^{\cG_\Qp}_{\cG_F}(\xi) \cong \ind^{\cG_\Qp}_{\cG_F}(\xi')$ for another continuous character $\xi'$ if and only if $\xi' = \xi$ or $\xi' = \xi^\sigma$. Now, if $|X(\psi)| > 2$, let $F/\Qp$ be such that $\kappa_{F/\Qp} \in X(\psi)$, and let $\chi \in X(\psi)$ be a non-trivial character different from $\kappa_{F/\Qp}$. After passing to the extension $E'$ we find $\psi \cong \chi \otimes \psi = \chi \otimes \ind^{\cG_\Qp}_{\cG_F}(\xi) = \ind^{\cG_\Qp}_{\cG_F}(\chi_F \xi)$, with $\chi_F = \chi|_{\cG_F}$. By what we have just observed, we then have $\chi_F \xi = \xi$ or $\chi_F \xi = \xi^\sigma$. As $\chi \neq \kappa_{F/\Qp}$ it follows that $\chi_F = \xi^\sigma/\xi$. Since $\cG_F$ is of index two in $\cG_\Qp$, it follows that there are two characters $\chi \in \cG_\Qp^*$ whose restriction to $\cG_F$ is $\xi^\sigma/\xi$, namely $\chi$ and $\kappa_{F/\Qp}\chi$, and thus $|X(\psi)|=4$ in this case. 
\end{proof}

\begin{remark}\label{defn-triply} Recall that a two-dimensional Galois representation is called {\it primitive} (resp. {\it simply-imprimitive}, resp. {\it triply-imprimitive}) if $X(\psi)$ has cardinality one (resp. two, resp. four).
\end{remark}

\begin{theorem}\label{Main}  Let $\psi: \cG_\Qp \to GL_2(E)$ be an absolutely irreducible continuous representation and $\Pi=\Pi(\psi)$.  Denote by $\opsi: \cG_\Qp \to PGL_2(E)$ the corresponding projective representation. Then, after possibly replacing $E$ by a finite extension, we have  

\begin{enumerate}
    \item[(i)]$\Pi|_H$ is either absolutely irreducible or it decomposes as a direct sum of two absolutely irreducible representations. Let $r$ be the number of components of $\Pi|_H$ (counting multiplicities).

    \item[(ii)]  $S_\opsi$ is finite, with cardinality 1,2, or 4.

    \item[(iii)] If $\psi$ is not triply-imprimitive, then  $\Pi_H$ is multiplicity free. Moreover,
    \[
        r=|S_\opsi| \leq 2.
    \]
    That is, the number of components of $\Pi|_H$
    equals the cardinality of $S_\opsi$.

    \item[(iv)] If $\psi$ is  triply-imprimitive, then  $|S_\opsi|=4$ and $\Pi$ restricts to $H$ with multiplicity two. More specifically,
    $\Pi|_H$  decomposes as a direct sum of two equivalent irreducible representations. 

\end{enumerate}

\end{theorem}

\begin{proof}

(i) This has already been shown in \ref{Prop2components}. 

(ii) Follows from Lemma~\ref{lem:X}(iii) and 
Proposition~\ref{prop:card}. Note that 
$|S_\opsi|=|X(\Pi)|$.

(iii)  Let $\pi$ be an irreducible subrepresentation of $\Pi|_H$.
After possibly replacing $E$ by a finite extension,
we may assume that $\pi$ is absolutely irreducible.
We apply Proposition~\ref{restrictionPiAbs}. 
Property (iv) tells us that $|X(\Pi)|=m^2|G:G_\pi|$, where $m$ is the multiplicity of $\pi$ in $\Pi|_H$.
Since $|X(\Pi)| \leq 2$, the multiplicity must be one and $|X(\Pi)|=|G:G_\pi|$.
Property (ii) then tells us that $\Pi|_H$ is a direct sum of $|X(\Pi)|$ inequivalent representations.   

(iv) If $\psi$ is  triply-imprimitive, then by definition $|X(\psi)|=4$ and hence $|S_\opsi|=4$.
The statement then follows from Corollary~\ref{cor:two}.   \end{proof}

\begin{proposition} Let $F/\Qp$ be a finite extension, and let $\psi: \cG_F \to \GL_2(E)$ be an absolutely irreducible 2-dimensional triply-imprimitive continuous representation. Then after possibly replacing $E$ by a finite extension, there is a continuous character $\tau$ such that $\tau \otimes \psi$ has finite image. In particular, $\tau \otimes \psi$ is de Rham with Hodge-Tate weights zero.  
\end{proposition}

\begin{proof} Let $F_2/F$ be a quadratic extension such that $\psi = \Ind^{\cG_F}_{\cG_{F_2}} (\xi)$. Here we consider $\xi$ as a character of the profinite completion $\widehat{F_2^\x}$ of the multiplicative group $F_2^\x$, which is isomorphic to the abelianized Galois group $\cG_{F_2}^\ab$. As is shown in the proof of \ref{prop:card}, one then
has $\xi/\xi^\sigma = \chi \circ N$, where $\sigma$ is the non-trivial Galois automorphism of $F_2$ over $F$, $N$ is the norm from $F_2$ to $F$, and $\chi$ is a character of $\widehat{F^\x}$. This means for any element $x$ in $\widehat{F_2^\x}$ one has 
\begin{numequation}\label{triply-eq0}
\xi(x/\sigma(x)) = \chi(x\sigma(x)) \;.
\end{numequation}
Let $i: \widehat{F^\x} \to \widehat{F_2^\x}$ be the canonical (injective) map induced by the inclusion $F^\x \to F_2^\x$, and define the character $\zeta$ of $\widehat{F^\x}$ by $\zeta(y) = \xi(i(y))$. If we multiply equation \ref{triply-eq0} by $\xi(x\sigma(x))$ we obtain
\begin{numequation}\label{triply-eq1}
\xi(x^2) = ((\chi \cdot \zeta) \circ N)(x)
\end{numequation}
It follows from \cite[II, 5.7 (i)]{Neukirch-ANT} that there is an isomorphism of topological groups 
\begin{numequation}\label{triply-eq2}
\widehat{F^\x} \cong \widehat{\Z} \oplus \Z/(q-1) \oplus \Z/(p^a) \oplus \bbZ_p^{\oplus d} \;, 
\end{numequation}
where $q$ is the number of elements in the residue field of $F$, $p^a$ is the number of roots of unity of $p$-power order in $F$, and $d = [F: \Qp]$. Choose a finite order character $\theta$ on $\widehat{F^\x}$ such that $\tau_1 := \chi \cdot \zeta \cdot \theta$ has the property $\tau_1(-1) = 1$. For example, we may choose $\theta$ to be the restriction of $\chi \cdot \zeta$ to the finite torsion subgroup of $\widehat{F^\x}$, which is a direct summand in $\widehat{F^\x}$ by \ref{triply-eq2}. Note also that $-1$ is the only non-trivial 2-torsion element in $\widehat{F^\x}$, which implies that if $y_1^2 = y_2^2$ in $\widehat{F^\x}$, then $y_1 = \pm y_2$. Then we define the character $\tau$ on $\widehat{F^\x}$ by $\tau(y) = \tau_1(y_1)$, where $y_1^2 = y^{-1}$. This is well-defined because $\tau_1(-1) = 1$. We thus have $\tau(y)^2 = \tau(y^2) = \tau_1(y^{-1}) = \tau_1(y)^{-1}$ for all $y \in \widehat{F^\x}$.
We tensor $\psi$ with $\tau$ and obtain
\begin{numequation}\label{triply-eq3}
\tau \otimes \psi = \tau \otimes \Ind^{\cG_F}_{\cG_{F_2}}(\xi) = \Ind^{\cG_F}_{\cG_{F_2}} ((\tau \circ N) \cdot \xi)
\end{numequation}
By \ref{triply-eq1} we find that for all $x \in \widehat{F_2^\x}$
\[((\tau \circ N) \cdot \xi)(x)^2 = \tau_1(N(x))^{-1}\xi(x^2) = \theta(N(x))^{-1}\] 
has finite order, and $(\tau \circ N) \cdot \xi$ is therefore itself a character of finite order. But then the induced representation in \ref{triply-eq3}, i.e., $\tau \otimes \psi$, has finite image.
\end{proof}

\subsection{Classification}

The classification of ordinary representations of $H$ is given in Proposition~\ref{prop:ordinary}. In this section, we
classify non-ordinary representations.

Let $\cG_2^0$ be the set of equivalence classes of absolutely irreducible continuous
representations $\psi: \cG_\Qp \to GL_2(\overline{\Qp})$.
(Any such representation has values in some $GL_2(E)$ with $E/\Qp$ finite.)
Let $\overline{\cG}_2^0$ be the   set of equivalence classes of the corresponding projective representations $\opsi: \cG_\Qp \to PGL_2(\overline{\Qp})$.

Suppose $\psi, \psi' \in \cG_2^0$ satisfy
$\opsi = \opsi'$. Then $\psi= \tau \otimes \psi'$, for some character $\tau$. It follows
$\Pi(\psi) = (\det \circ \tau) \otimes \Pi(\psi')$ and $\Pi(\psi)|_H = \Pi(\psi')|_H.$ Hence, 
the restriction 
$\Pi(\psi)|_H$ does not depend on the lift  $\psi$ of $\opsi$.

\begin{proposition}\label{prop:class}

Let $\pi$ be an 
absolutely irreducible non-ordinary admissible unitary Banach space representation of $H$.
Then there exists a unique 
$\opsi \in \overline{\cG}_2^0$
such that $\pi$ is equvivalent to
a constituent of $\Pi(\psi)|_H$, where $\psi \in \cG_2^0$ is any lift of $\opsi$. More specifically,

\begin{enumerate}

    \item[(1)] If $\psi$ is primitive, then 
    $\pi=\Pi(\psi)|_H$.
 
    \item[(2)] If $\psi$ is simply-imprimitive, then $\pi$ is isomorphic to one of the two non-isomorphic direct summands of $\Pi(\psi)|_H$.  

    \item[(3)] If $\psi$ is triply-imprimitive, then $\pi$ is isomorphic to the unique irreducible representation  which appears with multiplicity two in  $\Pi(\psi)|_H$.
\end{enumerate}

\end{proposition}

\begin{proof} Let $\pi$ be any absolutely irreducible admissible unitary Banach space representation of $H$. By \ref{prop:induced} there is an irreducible admissible unitary Banach space representation $\Pi$ of $G$, which we may assume to be absolutely irreducible, such that $\Pi|_H$ contains $\pi$. If $\Pi$ is ordinary, then so is $\pi$, because the restriction of a parabolically induced representation is parabolically induced (cf. the proof of \ref{prop:irred}). Hence, if $\pi$ is not ordinary, then $\pi$ is a constituent of a representation of the form $\Pi(\psi)$.

\vskip8pt

To see that the representations listed in (1)-(3)
are not ordinary and no two of them are isomorphic, 
we first recall that $\psi \mapsto \Pi(\psi)$
is a bijection between equivalence classes and that $\Pi(\psi)$ is non-ordinary. Then, we 
argue as follows.
Suppose we have an isomorphism 
$$\pi_1 \stackrel{\simeq}{\lra} \pi_2,$$
where 
$\pi_1$ and $\pi_2$ are absolutely irreducible admissible unitary Banach space
representations of $H$ and $\pi_2$ is not ordinary. As explained above,
we have $\pi_i \hra \Pi_i|_H$  for some 
irreducible admissible unitary Banach space
representation $\Pi_i$ of $G$.

We extend the central characters of $\pi_1$ and $\pi_2$ to characters of the center $Z$ of $G$, and obtain an isomorphism $\pi_1 \stackrel{\simeq}{\lra} \pi_2$ as representations of $ZH$. Inducing from $ZH$ to $G$ gives an isomorphism $\ind^G_{ZH}(\pi_1) \stackrel{\simeq}{\lra} \ind^G_{ZH}(\pi_2)$. By \ref{restrictionPiAbs}, $\Pi_1$ is a direct summand of $\ind^G_{ZH}(\pi_1)$, and $\ind^G_{ZH}(\pi_2)$ is a direct sum of representations of the form $\chi \otimes \Pi_2$. Hence the composition 
\[
    \Pi_1 \hra \ind^G_{ZH}(\pi_1) \stackrel{\simeq}{\lra} \ind^G_{ZH}(\pi_2) \twoheadrightarrow \chi \otimes \Pi_2 
\] 
maps $\Pi_1$ isomorphically to at least one of the direct summands $\chi \otimes \Pi_2$.
We thus obtain an isomorphism of $G$-representations $\Pi_1 \stackrel{\simeq}{\lra} \chi \otimes \Pi_2$.
This is impossible if 
$\Pi_1$ is ordinary, because $\Pi_2$
is non-ordinary.

Assume that $\Pi_1$ is non-ordinary.
Then $\Pi_1 \cong \Pi(\psi_1)$ and 
$\Pi_2 \cong \Pi(\psi_2)$, for some (unique)
$\psi_1, \psi_2 \in \cG_2^0$.
Note that $\chi = \tau \circ \det$ for some quadratic character $\tau$ of $\bbQ_p^\x$. 
Then 
\[
   \Pi(\psi_1) \cong \Pi_1 \cong  (\tau \circ \det) \otimes \Pi_2 \cong \Pi(\tau \otimes \psi_2).
\]
It follows $\psi_1 \cong \tau \otimes \psi_2$ and hence $\opsi_1 \cong \opsi_2$.
\end{proof}

\section{Connection with the classical local Langlands correspondence}\label{sec:LLC}

\subsection{Locally algebraic vectors}\label{sec:LocAlg}

 Let $\psi: \cG_\Qp  \to GL_2(E)$ be an absolutely irreducible representation and 
$\Pi=\Pi(\psi)$. The space of locally algebraic vectors 
$\Pi^\lalg$ is non-zero if and only if $\psi$ is de Rham with distinct Hodge-Tate weights $a<b$
\cite[Theorem 0.20]{Colmez10}.
In this case, by \cite[VI.6.50]{Colmez10}, $\Pi^\lalg$ decomposes as
\[
  \Pi(\psi)^\lalg = \Pi(\psi)^\alg \otimes \s(\psi),
\]
where $\Pi(\psi)^\alg \cong \det^a \otimes_E \Sym^{b-a-1}(E^2)$ is absolutely irreducible, and $\s(\psi)$ is a smooth representation.

In this section, we assume that 
$\psi: \cG_\Qp \to GL_2(E)$ is de Rham with distinct Hodge-Tate weights $a<b$.

There exists a finite Galois extension 
$L/\Qp$ such that $\psi$ becomes semistable when restricted to $\cG_L=\Gal(\overline{\Q}_p/L)$.
If $\psi$ is crystalline or semistable, then one may take $L=\Q_p$. In any case, we assume that $E$ contains $L$, and 
$
    |\Hom_{\Q_p}(L,E)| = |\Gal(L/\Qp)|.
$
There is an equivalence of categories $D_{\st}^L$ between the category of $E$-representations of $\cG_\Qp$ which become semistable when restricted to 
$\cG_L$ and the category of weakly admissible  filtered $(\varphi,N, L,E)$-modules \cite[Corollary 2.10]{Savitt}.
We refer to  \cite[Section 2]{Ghate-Mezard} or
\cite[Section 2]{Savitt}
for the definition of a weakly admissible $(\varphi,N, L,E)$-module $D$. It is a free
$(L_0 \otimes_\Qp E)$-module of finite rank, where $L_0$
is the maximal unramified extension of $\Qp$
contained in $L$. It is equipped with the Frobenius endomorphism $\vp$, the monodromy operator $N$, the action of $\Gal(F/\Qp)$,
and  a filtration on $D_L = L \otimes_{L_0} D$.

Let $D=D_{\st}^L(\psi)$ and let
$Fil(D_L)$ be the filtration on $D_L$ associated to $\psi$.
The jumps in the filtration are the negatives of the Hodge-Tate weights, so 
\[
    Fil^i(D_L) = \begin{cases}
       D_L & \text{ if } \, i \leq -b,  \\
       Fil^{-a}(D_L) & \text{ if } \, -b < i \leq -a, \\
       0 & \text{ if } \, i > -a.
    \end{cases}
\]

Using the decomposition $L_0 \otimes_\Qp E = \prod_{\Hom_{\Q_p}(L_0,E)} E$ one obtains from $D$ the 
corresponding Weil-Deligne representation
$\WD(D)$, also denoted by $\WD(\psi)$
(see section 2.4 of \cite{Ghate-Mezard}).
This is a two-dimensional $E$-vector space 
with an induced action of $(W_\Qp,N)$.
Let $$\phi=\WD(\psi)^{F{\rm -s.s.}}$$ be the  $F$-semisimplification of $\WD(D)$, in the sense of
\cite[8.5]{Deligne_fonctions_L}. 
We denote by $\ophi$ the associated projective 
Weil-Deligne representation.
We set $\cS_\ophi = S_\ophi/S_\ophi^\circ$.

\begin{remark}\label{rem:classic}
Let $\phi=\WD(\psi)^{F{\rm -s.s.}}$, and let $\s(\phi)$ be the smooth representation of $\GL_2(\Qp)$ associated to $\phi$ by the classical local Langlands correspondence.
Up to normalization, $\s(\psi)$ is equal to $\s(\phi)$, except if the latter representation would be 1-dimensional, in which case $\s(\psi)$ is the unique principal series representation which surjects onto this character, cf. \cite[before VI.6.50]{Colmez10} for details.\footnote{$\s(\psi)$ is denoted ${\rm LL}({\rm WD}(\psi))$ in loc. cit.}
For normalization, see \cite[Section 4]{BrSch}.
\end{remark}

\begin{lemma}\label{lem:decomp}

Let  $\Pi=\Pi(\psi)$ and $\s=\s(\psi)$.
If $\Pi|_H$ is decomposable, then $\s|_H$
is also decomposable.

\end{lemma}

\begin{proof}  If $\Pi|_H = \Pi_1 \oplus \Pi_2$, then both $\Pi_1$ and $\Pi_2$ contain locally algebraic vectors. Since the algebraic representation $\Pi^\alg$ remains  irreducible when restricted to $H$,
$\s|_H$ must be decomposable.
\end{proof}

\begin{remark}\label{rem:nonTrian}
Let  $\Pi=\Pi(\psi)$ and $\s=\s(\psi)$.
If $\psi$ is de Rham but not trianguline over any finite extension of $E$\footnote{In fact, if $\psi$ becomes trianguline after base change to a finite extension of $E$, then it does so for a quadratic extension of $E$, by \cite[4.17]{Colmez-triang}.}, then $\psi$ is neither cristabelline, by \cite[4.16, 4.17]{Colmez-triang}, nor semistable, by \cite[4.23]{Colmez-triang}.  By \cite[VI.6.50]{Colmez10}, the smooth representation $\s$ is supercuspidal. These references show also that the converse is true: if $\psi$ is absolutely irreducible and $\s$ is  supercuspidal, then $\psi$ is de Rham and it is not trianguline over any finite extension of the coefficient field $E$.
\end{remark}

\subsection{Trianguline case}

\begin{lemma}\label{lem:triang}

Suppose that 
$\psi: \cG_\Qp \to GL_2(E)$ is trianguline with distinct Hodge-Tate weights.
Let $\phi = \WD(\psi)^{F{\rm -s.s.}}$.
Then $|S_\opsi|=|\cS_\ophi| \leq 2.$

\end{lemma}

\begin{proof} Let $\Pi=\Pi(\psi)$ and $\s=\s(\psi)$. It follows from Theorem~\ref{Main} 
that  $|S_\opsi|\leq 2$ and that 
$\Pi|_H$ decomposes if and only if $|S_\opsi|=2$.
Similarly, we know from the classical Langlands correspondence
that $|\cS_\ophi|\leq 2$ and that 
$\s|_H$ decomposes if and only if $|\cS_\ophi|=2$.
Lemma~\ref{lem:decomp} then gives  $|S_\opsi|\leq |\cS_\ophi|$.
It remains to consider the case  
$|\cS_\ophi|=2$ and show that $|S_\opsi|=2$.

By Remark~\ref{rem:nonTrian}, $\s$ is not supercuspidal. Since $\s|_H$ decomposes, it is known from the representation theory of smooth representations that 
$\s = \ind_P^G(\a \otimes \beta |\, |^{-1})$,
where $\a, \beta$ are smooth characters of $\Q_p^\times$ and $\chi=\a\beta^{-1}$ is a quadratic
character.

On the other hand, we can find an explicit description of $D=D_{\st}^L(\psi)$ in 
\cite{BB}.
It is equal to $D(\a, \beta)$, as described 
in Definition 2.4.4 of loc.cit.
It can be shown by direct computation with matrices that $\beta = \chi \a$ and  $\chi^2=1$ imply 
$D \cong \chi \otimes D.$
Then $\psi \cong \chi \otimes \psi$.
It follows $\chi \in X(\psi)$ and by Lemma~\ref{lem:X} (iii), $|S_\opsi| \geq 2$.
Then of course $|S_\opsi| = 2$.
\end{proof}

\begin{remark}
For $\psi: \cG_\Qp \to GL_2(E)$ crystabelline, 
the centralizer $S_\ophi$ is infinite, while $S_\ophi(\psi)$
is a finite subgroup of $S_\ophi$ equivalent to the component group $\cS_\ophi$.
For instance, with notation as in the proof of the previous lemma, take $\beta = \chi \a$ with $\chi^2=1$, so $D \cong \chi \otimes D.$
In this case, $S_\ophi$ consists of diagonal and anti-diagonal matrices, while $S_\ophi(\psi)$ has two elements: $\bar 1$ and $\bar g$, where $g$ is an anti-diagonal matrix depending on the filtration $Fil(\psi)$.

\end{remark}

\begin{corollary}\label{cor:trian}
Suppose $\psi: \cG_\Qp \to GL_2(E)$  is trianguline.
Let $\Pi=\Pi(\psi)$.
Then $\Pi|_H$ is decomposable if and only if $(\Pi^\lalg)|_H$ is decomposable. In this case, $\Pi|_H$ has two irreducible non-isomorphic constituents $\pi_1$ and $\pi_2$. Furthermore, the representations $(\pi_1)^\lalg$ and $(\pi_2)^\lalg$ are the irreducible constituents of $(\Pi^\lalg)|_H$, and they are not isomorphic.  

\end{corollary}

\subsection{Non-trianguline case}

Suppose  $\psi:  \cG_\Qp \to GL_2(E)$ is non-trianguline de Rham with distinct Hodge-Tate weights $a<b$, and continues to be non-trianguline after base change to any finite extension of $E$.
Then $\phi = \WD(\psi)^{F{\rm -s.s.}}$ is irreducible,
and hence $\WD(\psi)$ is irreducible as well, an thus must be $F$-semisimple, i.e., 
\[
    \phi=\WD(\psi) =\WD(D) \,.
\]

From \cite[Lemma 2.1]{Savitt}, we know that   $Fil^{-a}(D_L)$ is a free $(L \otimes_\Qp E)$-module 
of rank 1. Using the decomposition $L \otimes_\Qp E = \prod_{\Hom_{\Q_p}(L,E)} E$ one obtains from $Fil^{-a}(D_L)$ a
1-dimensional $E$-subspace 
$$
    Fil^{-a}(\psi)
$$
of $\WD(\psi)$.\footnote{We can also define the corresponding filtration $Fil(\psi)$
on $\WD(\psi)$, but it will not play a role in our computation.}
We denote by $\ophi$ the associated projective 
Weil-Deligne representation. Given $\bar{g} \in S_\opsi \subset PGL_2(\overline{E})$, any lift $g \in GL_2(\overline{E})$ of $\bar{g}$ is naturally an isomorphism $g: \psi \to \vep_{\bar{g}} \otimes \psi$ with a quadratic character $\vep_{\bar{g}}$ of $\cG_\Qp$ which only depends on $\bar{g}$ and not the lift $g$, cf. the proof of the lemma below. This isomorphism in turn induces an isomorphism $\WD(g): \phi = \WD(\psi) \to \WD(\vep_{\bar{g}} \otimes \psi) = \vep_{\bar{g}} \otimes \WD(\psi)$, and can thus be considered as an automorphism of the vector space $\WD(\psi)$ underlying $\phi$, i.e., as an element of $GL(\WD(\psi))$. We can then consider the corresponding automorphism modulo scalars $\overline{\WD(g)} \in PGL(\WD(\psi))$, which does not depend on the lift $g$ and we henceforth denote by $\WD(\bar{g})$. This automorphism commutes then with the image of $\phi$ and is therefore an element of the centralizer $S_\ophi$ in $PGL(\WD(\psi))$ of the image of $\ophi$. It is clear that this construction furnishes a group homomorphism
\begin{numequation}\label{WDhom}
    \WD: S_\opsi \to S_\ophi \;.
\end{numequation}

\begin{lemma}\label{lem:SphiPsi}
The homomorphism \ref{WDhom} is injective and its image is the subgroup $S_\ophi(\psi)$ of
$ S_\ophi$ consisting of elements which preserve the subspace $Fil^{-a}(\psi)$ of $\WD(\psi)$.
\end{lemma}

\begin{proof} 

Let $\bar{g} \in S_\opsi \subset  PGL_2(\ovE)$.
Then
$
\bar{g} \bar{\psi}(x) \bar{g}^{-1} = \bar{\psi}(x)
$
for all $x \in \cG_\Qp$.
If we represent $\bar{g}$ by an element $g$ in $GL_2(E)$, we find that
\[
g \psi(x) g^{-1} = \vep_{\bar{g}}(x) \psi(x)
\]
for some scalar $\vep_{\bar{g}}(x) \in \ovE^\times$. Then, as usual, it follows that $\vep_{\bar{g}}$ is a
homomorphism $\cG_\Qp \to \ovE^\times$, and $\vep_{\bar{g}}(x)^2 = 1$ for all $x$, i.e.,
$\vep_{\bar{g}}$ is a quadratic character on $\cG_\Qp$.
Therefore, $g$ is an isomorphism
$g:\psi \to \vep_{\bar{g}} \otimes \psi$.
It induces an isomorphism $D(g)$ on $D=D_\st^L(\psi)$
which preserves the filtration $Fil(D_L)$. And this isomorphism in turn induces the isomorphism
$WD(g): \phi \to \vep_{\bar{g}} \otimes \phi$ which preserves
the subspace $Fil^{-a}(\psi)$.
Conversely, if we have an isomorphism $h: \phi \to \vep_{\bar{g}} \otimes \phi$ which commutes with the Weil group action and monodromy operator, and which preserves the  subspace $Fil^{-a}(\psi)$, 
it induces an isomorphism on $D$ (\cite[Proposition 4.1]{BrSch}) which preserves the filtration $Fil(D_L)$,
so it comes from a (unique) isomorphism $g:\psi \to \vep_{\bar{g}} \otimes \psi$.
\end{proof}

Suppose $\phi$ is imprimitive. Then
\begin{numequation}\label{psiInd}
    \phi = \ind^{W_\Qp}_{W_F} \chi,
\end{numequation}
where $F/\Qp$ is a quadratic extension and $\chi$ is a (smooth)
character of $W_F$, i.e., by local class field theory,
a smooth character of $F^\times$. 

Let $\s \in W_\Qp-W_F$, so that $W_\Qp = W_F \coprod W_F\s$.
Let $\chi^s$ be the conjugate character: 
$
    \chi^s(x) = \chi(\bar{x}),
$
where
$
    x \mapsto \bar{x}
$
is the non-trivial Galois automorphism. Define $f_0, f_1: W_\Qp \to E$  by
\begin{numequation}\label{f0f1} 
    f_0(y) = \begin{cases}  \chi(y), \quad & \text{if } y \in W_F,\\
      0, & \text{if } y \in W_F \sigma,
      \end{cases}
\qquad \quad
    f_1(y) = \begin{cases}  0, \quad & \text{if } y \in W_F,\\
      \chi(y \s^{-1}), & \text{if } y \in W_F \sigma.
      \end{cases}
\end{numequation}
These two elements form a basis of $\phi$.
For $(\a, \b) \in E^2 \setminus \{(0,0)\}$, let us denote by $Fil_{\a,\b}$ the one dimensional subspace spanned by $\a f_0 + \b f_1$. 
We write $[\a:\b]$ for the equivalence class of 
$(\a, \b)$ in the projective plane.
Recall that $Fil^{-a}(\psi)$ is
a one-dimensional subspace of $\phi$.
Hence, it is of the form
$Fil^{-a}(\psi)=Fil_{\a,\b}$ for some  $[\a:\b]$.

\begin{proposition}\label{S_opsi}

Suppose  $\psi:  \cG_\Qp \to GL_2(E)$ is non-trianguline de Rham with distinct Hodge-Tate weights $a<b$, and continues to be non-trianguline after base change to any finite extension of $E$.
 Let $\phi = WD(\psi)$  and let $Fil^{-a}(\psi)=Fil_{\a,\b}$. 

\begin{enumerate}

    \item[(a)] If $\phi$ is primitive, then $S_\opsi = S_\ophi =1$.
    
    \item[(b)] Suppose $\phi$ is imprimitive and write
    $\phi = \ind^{W_\Qp}_{W_F} \chi$
as in \eqref{psiInd}. Then

\begin{enumerate}

    \item[(i)] $|S_\opsi| \leq 2$, and $|S_\opsi| \leq |S_\ophi| \leq 4$. 
    
    \item[(ii)] If  $(\chi/\chi^s)^2 \ne 1$, then $|S_\ophi| = 2$. Moreover, $|S_\opsi| = 2$ if and only if $[\a:\b]=[1:0]$ or $[0:1]$.
    
    \item[(iii)] If  $(\chi/\chi^s)^2 = 1$, then $S_\ophi = \Z / 2\Z \oplus \Z / 2Z$. Moreover, $|S_\opsi| = 2$ if and only if $[\a:\b]$ is equal to one of the following
    \[
       [1:0], \, [0:1], \, [1:\gamma], \, [1:-\gamma], \, [1:\delta], \, or \, [1:-\delta],
    \]
    where $\gamma^2=\chi(\sigma)^2$ and $\delta^2=-\chi(\sigma)^2$.
\end{enumerate}

\end{enumerate}

\end{proposition}

\begin{proof} 
We know from Lemma~\ref{lem:SphiPsi} that
$ S_\opsi \cong S_\ophi(\psi)$.
Hence, our proof reduces to finding all
$h \in GL_2(E)$ such that
\begin{numequation}\label{Eq*}
    h \phi(x) h^{-1} = \vep_{\bar{h}}(x) \phi(x)
\end{numequation}
for all $x$ in the Weil group $W_\Qp$
and such that $h$ fixes  $Fil_{\a,\b}$.
Then we can identify  $S_\opsi$ with the  subgroup of $S_\ophi$ consisting of all elements fixing  $Fil_{\a,\b}$.
Now, (a) follows from \cite{Shelsted_notes}, case (iii) in section 11 ($\ophi$ tetrahedral or octahedral).

For (b), suppose $\phi = \ind^{W_\Qp}_{W_F} \chi$.
Define the basis $\{f_0, f_1\}$ as in \eqref{f0f1}.
 If $x$ is in $W_F$, then the matrix of $\phi(x)$ with respect to this basis is
\[
\diag(\chi(x),\chi^s(x)).
\]
Furthermore, we have for $y \in W_F$:
\[
(\s.f_0)(y\s) = f_0(y\s^2) = \chi(\s^2)\chi(y) = \chi(\s^2) f_1(y\s)
\]
and thus $\s.f_0 = \chi(\s^2) f_1$ and $\s.f_1 = f_0$, so the matrix of $\s$
with respect to $\{f_0,f_1\}$ is
the anti-diagonal matrix 
\[
    A = \left( \begin{matrix}  0 & 1 \\ 
                        \chi(\s^2) & 0
            \end{matrix} \right).
\]
If $h$ satisfies \eqref{Eq*}, then $h$ must permute the eigenspaces of the operators
$\phi(y)$ for $y\in W_F$, and
hence must be diagonal or anti-diagonal.

Suppose $h$ is a diagonal matrix, $h = \diag(a,d)$. Then $\vep_{\bar{h}}(x) = 1$, for all $x\in W_F$. Furthermore, from $h A h^{-1} = \vep_{\bar{h}}(\s) A$ we have
\[
    \left( \begin{matrix}  0 & a d^{-1}\\ 
                        da^{-1}\chi(\s^2) & 0
            \end{matrix} \right)
    = \vep_{\bar{h}}(\s) \left( \begin{matrix}  0 & 1 \\ 
                        \chi(\s^2) & 0
            \end{matrix} \right).
\]
It follows $d=a$ or $d=-a$.
If $d=a$, then $h=\diag(a,a)$ and $\bar{h}=\bar{1}$.

Set $h_1=\diag(a,-a)$ and take 
$\diag(1,-1)$ as a representative of $\bar{h}_1.$
Then
\[
   \left( \begin{matrix}  1 & 0\\ 
                        0 & -1
            \end{matrix} \right) 
       \left( \begin{matrix}  \a \\ \b \end{matrix} \right) 
   =
         \left( \begin{matrix}   \a \\ - \b \end{matrix} \right) .
\]
This is a scalar multiple of $\left( \begin{matrix}  \a \\ \b \end{matrix} \right)$ 
if and only if $\a=0$ or $\b=0$. 

Next, suppose $h$ is anti-diagonal, 
$
    h = \left( \begin{matrix}  0 & b\\ 
                        c & 0
            \end{matrix} \right). 
$
From \eqref{Eq*}, for $x \in W_F$ we have
\[
   h \left( \begin{matrix}  \chi(x) & 0\\ 
                       0 & \chi^s(x)
            \end{matrix} \right) h^{-1}
    = \left( \begin{matrix}  \chi^s(x) & 0\\ 
                     0 &  \chi(x) 
            \end{matrix} \right)
    = \vep_{\bar{h}}(\s) \left( \begin{matrix}  \chi(x) & 0 \\ 
                       0 & \chi^s(x)
            \end{matrix} \right).
\]
This is possible only for $(\chi/\chi^s)^2=1$.
Hence, if $(\chi/\chi^s)^2 \ne 1$, then $S_\ophi = \{ \bar{1},\bar{h}_1 \} $, where $h_1=\diag(1,-1)$. This completes the proof of (ii).

We continue the proof for $(\chi/\chi^s)^2=1$. Let $h$ be the anti-diagonal matrix as above. 
From $h A h^{-1} = \vep_{\bar{h}}(\s) A$, we have
\[
    \left( \begin{matrix}  0 & b c^{-1} \chi(\s^2)\\ 
                        c b^{-1} & 0
            \end{matrix} \right)
    = \vep_{\bar{h}}(\s) \left( \begin{matrix}  0 & 1 \\ 
                        \chi(\s^2) & 0
            \end{matrix} \right).
\]
It follows $c = \chi(\s^2) b$ or $c = -\chi(\s^2) b$.
We may take $b=1$.
Set
\[
    h_2 =  \left( \begin{matrix}  0 & 1\\ 
                      \chi(\s^2) & 0
            \end{matrix} \right)
            \quad \text{and} \quad
    h_3 =  \left( \begin{matrix}  0 & 1\\ 
                      -\chi(\s^2) & 0
            \end{matrix} \right).
\]
Then 
$
    S_\ophi = \{ \bar{1},\bar{h}_1,\bar{h}_2, \bar{h}_3 \} \cong
    \Z / 2\Z \oplus \Z / 2Z.
$

Assume that $h_2$ fixes  $Fil_{\a,\b}$. Then
\[
   \left( \begin{matrix}  0 & 1\\ 
                        \chi(\s^2) & 0
            \end{matrix} \right) 
             \left( \begin{matrix}   \a \\ \b \end{matrix} \right) 
   =    \left( \begin{matrix}  \b \\ \chi(\s^2) \a  \end{matrix} \right) 
   =
       \gamma  \left( \begin{matrix}   \a \\ \b \end{matrix} \right) .
\]
It follows $\beta = \gamma \a$ and $\gamma^2 = \chi(\sigma^2)$,
and $[\a:\b]= [1: \pm \gamma]$.
Similarly, for $h_3$
we get $[\a:\b]= [1:\delta]$ or $[1:-\delta]$, where 
$\delta^2 = -\chi(\sigma^2)$.
\end{proof}

From Proposition~\ref{S_opsi}, it is easy to describe the relation between $\Pi|_H$
and $\s|_H$. For simplicity, we 
assume in the corollary below that the Hodge-Tate weights are 0 and 1, so that $\s$ is the space of smooth vectors in $\Pi$.
The statement can be easily modified to a general case by replacing smooth vectors with locally algebraic vectors.

\begin{corollary}\label{cor:nontrian}

Suppose  $\psi:  \cG_\Qp \to GL_2(E)$ is non-trianguline de Rham with  Hodge-Tate weights 0 and 1, and continues to be non-trianguline after base change to any finite extension of $E$.
Let $\Pi=\Pi(\psi)$ and $\s=\s(\psi)$.
Then, after possibly replacing $E$ by a finite extension, we have

\begin{enumerate}

    \item[(a)] $\s$ is supercuspidal.

    \item[(b)] Both $\Pi|_H$ and $\s|_H$ are multiplicity free.
    
    \item[(c)] If $\phi$ is primitive, then both $\Pi|_H$ and $\s|_H$ are irreducible.

    \item[(d)] If $\phi$ is simply-imprimitive, then either

\begin{enumerate}

    \item[(i)] $\Pi|_H=\pi$ is irreducible and $\pi^\sm=\s|_H$ decomposes as a direct sum of  two  components, or

    \item[(ii)] $\Pi|_H=\pi_1 \oplus \pi_2$
    and $\s|_H= \pi_1^\sm \oplus \pi_2^\sm$, with each $\pi_i^\sm$ irreducible.

\end{enumerate}

    \item[(e)] If $\phi$ is triply-imprimitive, then either

\begin{enumerate}

    \item[(i)] $\Pi|_H=\pi$ is irreducible and $\pi^\sm=\s|_H$ decomposes as a direct sum of four components, or

    \item[(ii)] $\Pi|_H=\pi_1 \oplus \pi_2$
    and $\s|_H= \pi_1^\sm \oplus \pi_2^\sm$,
    where each $\pi_i^\sm$ decomposes as  a direct sum of  two  irreducible components, $\pi_i^\sm=\s_{i,1} \oplus \s_{i,2} $.

\end{enumerate}

\end{enumerate}

\end{corollary}

We remark that the case of $\psi$ triply-imprimitive de Rham is not included in Proposition~\ref{S_opsi}. In that case, the Hodge-Tate weights are equal and $\Pi(\psi)^\lalg=0$.

\begin{proof} We know from the classical local Langlands correspondence that $\s$ is supercuspidal, $\s|_H$ is multiplicity free,
and the number of components in $\s|_H$ is equal to $|S_\ophi|$. Similarly, since 
$\psi$ is not triply-imprimitive, we know from Theorem~\ref{Main} that  $\Pi|_H$ is multiplicity free
and the number of components in $\Pi|_H$ is equal to $|S_\opsi|$.
Then the simple comparison of $|S_\ophi|$ 
and $|S_\opsi|$, as described in Proposition~\ref{S_opsi}, together with Proposition~\ref{restrictionPialg}, implies all except (e)(ii). For (e)(ii), we need an additional observation that if $\Pi|_H=\pi_1 \oplus \pi_2$, then $\pi_1$ and $\pi_2$ are conjugate by some $g \in G$, by Proposition~\ref{restrictionPi}. Then 
$\pi_1^\sm$ and $\pi_2^\sm$ are also conjugate by $g$, and hence they must decompose as stated.
\end{proof}

\bibliographystyle{plain}
\bibliography{RefList}
\end{document}